\documentclass{article}%
\usepackage{amssymb}
\usepackage{amsfonts}
\usepackage{graphicx}
\usepackage{amsmath}%
\providecommand{\U}[1]{\protect\rule{.1in}{.1in}}
\newtheorem{theorem}{Theorem}

\newtheorem{corollary}[theorem]{Corollary}

\newtheorem{example}[theorem]{Example}

\newtheorem{lemma}[theorem]{Lemma}

\newtheorem{proposition}[theorem]{Proposition}
\newtheorem{remark}[theorem]{Remark}

\newenvironment{proof}[1][Proof]{\textbf{#1.} }{\ \rule{0.5em}{0.5em}}
\begin{document}
\bigskip{\LARGE Invariant subspaces of algebras of analytic elements
associated with periodic\ flows on W*-algebras}

\bigskip

\ \ \ \ \ \ \ \ \ \ \ \ \ \ \ \ \ \ \ \ \ \ \ \ \ \ \ \ \ \ \ \ \ \ \ \ \ \ \ \ Costel
Peligrad\footnote{Supported in part by a Taft Research Grant}

\bigskip

Department of Mathematical Sciences, University of Cincinnati, PO Box 210025,
Cincinnati, OH 45221-0025, USA. E-mail address: peligrc@ucmail.uc.edu

\bigskip


\bigskip


\bigskip

ABSTRACT. We consider an action of the circle group, $\boldsymbol{T}$ on a
$\sigma-$finite\ W*- algebra, $M.$ Similarly to the case when $M=L^{\infty
}(\boldsymbol{T})$ is acted upon by translations, we define the generalized
Hardy space $H_{+}\subset H$ where $H$\ is the Hilbert space of a standard
representation of $M$\ and the subalgebra $M_{+}$ of analytic elements of $M$
with respect to the action$.$ We prove that $M_{+}\subset B(H_{+})$ is a
reflexive algebra of operators if the Arveson spectrum is finite or, if the
spectrum is infinite, the spectral subspace corresponding to the least
positive element contains an unitary operator. We also prove that $M_{+}$\ is
reflexive if $M$\ is an abelian W*-algebra. Examples include the algebra of
analytic Toeplitz operators, $w^{\ast}$-crossed products, reduced $w^{\ast}%
$-semicrossed products and some reflexive nest subalgebras of von Neumann algebras.

\bigskip

\section{\bigskip Introduction}

Let $A\subset B(X)$ denote a weakly closed algebra of operators on the Hilbert
space $X$. Denote by $Lat(A)$ the lattice of closed subspaces of $X$\ that are
invariant for all operators $a\in A.$ Let $algLat(A)=\left\{  b\in
B(X)|bK\subset K\text{ for all }K\in Lat(A)\right\}  .$ The algebra \ is
called reflexive if $A=algLat(A)$ (for more information on reflexive operator
algebras, we recommend the book by H. Radjavi and P. Rosenthal [7]). Thus, a
reflexive operator algebra is completely determined by the lattice of its
invariant subspaces.The problem of whether a weakly closed algebra of
operators is reflexive started to be studied in the 1960s. In particular,
Sarason [8], proved that that the algebra of analytic Toeplitz operators on
the Hardy space $H^{2}(\boldsymbol{T})$ where $\boldsymbol{T}$\ is the the
unit circle $\boldsymbol{T}=\left\{  z\in%
\mathbb{C}
|\left\vert z\right\vert =1\right\}  ,$ is reflexive and so is any of its
weakly closed subalgebras. In [6] we extended this result to the case of
$H^{p}(\boldsymbol{T}),1<p\leqq\infty.$ The same paper contains a study of
reflexivity of subalgebras of analytic elements of crossed products of finite
von Neumann algebras on non commutative Hardy spaces $H^{p}$ associated with
the W*-dynamical system. Further, in [3], E. Kakariadis has considered the
more general case of semi crossed products and, among other results, he has
extended the particular case of our main result in [6] for $p=2$\ to the
semicrossed product setting. In this paper we study a related problem in a
more general setting than the crossed product or the reduced $w^{\ast}%
$semicrossed product considered in [6] and [3]. We consider a W*-dynamical
system $(M,\boldsymbol{T,}\alpha)$ where $M$\ is a $\sigma-$ finite W*-algebra
and a standard covariant representation of the system on a certain Hilbert
space, $H$, as constructed below. We then define the non commutative Hardy
space, $H_{+}$ associated with this representation. We show that if the
Arveson spectrum of the action is finite, then the algebra of analytic
elements, $M_{+}$, is a reflexive operator algebra on $H_{+}$ (Theorem 5). We
consider next the case of infinite Arveson spectrum and show that if the
spectral subspace corresponding to the smallest positive element of the
spectrum contains a unitary element, then, again, the algebra $M_{+}$ is
reflexive (Theorem 8). Finally, we consider the case of abelian von Neumann
algebras and we prove that in this case $M_{+}$ is always reflexive (Theorem
13). Examples and particular cases of our results include the classical result
of Sarason, the crossed product [6], the reduced $w^{\ast}-$semicrossed
product [3], the algebra of $n$ x $n$ upper triangular matrices and other nest
subalgebras of von Neumann algebras.

\section{\bigskip Notations and preliminary results}

Let $(M,\boldsymbol{T,}\alpha)$ be a W*-dynamical system, where $M$ is a
$W^{\ast}-$algebra with separable predual, $\boldsymbol{T}$\ is the circle
group, $\boldsymbol{T}=\left\{  z\in%
\mathbb{C}
|\left\vert z\right\vert =1\right\}  $\ and $\alpha\ $a $w^{\ast}-$continuous
action of $\boldsymbol{T}$\ on $M.$ For each $n\in%
\mathbb{Z}
,$\ denote by
\[
M_{n}=\left\{
{\displaystyle\int}
\overline{z}^{n}\alpha_{z}(m)dz|m\in M\right\}
\]
where the integral is taken in the $w^{\ast}-$topology\ Therefore,%
\[
M_{-n}=\left\{  m^{\ast}|m\in M_{n}\right\}
\]
It can immediately be checked that
\[
M_{n}=\left\{  m\in M|\alpha_{z}(m)=z^{n}m\right\}
\]
and%
\[
M_{-n}=(M_{n})^{\ast},\text{ where }(M_{n})^{\ast}=\left\{  m^{\ast}|m\in
M_{n}\right\}
\]
\ In particular, $M_{0}$\ is the algebra of the fixed points of $M$\ under the
action $\alpha.$ It is clear that the mapping $P_{n}^{M}:M\rightarrow M_{n}$
defined by $P_{n}^{M}(m)=%
{\displaystyle\int}
\overline{z}^{n}\alpha_{z}(m)dz$ is a projection of $M$\ onto the closed
subspace $M_{n}\subset M.$ Thus $P_{0}^{M}:M\rightarrow M_{0}$ is a faithful
normal conditional expectation of $M$\ onto the subalgebra $M^{\alpha}.$ It is
well known that $M$\ is the $w^{\ast}$-closed linear span of $\left\{
M_{n}|n\in%
\mathbb{Z}
\right\}  .$The set $\left\{  n\in%
\mathbb{Z}
|M_{n}\neq\left\{  0\right\}  \right\}  $ is called the Arveson spectrum of
the action $\alpha$ and will be denoted by $sp(\alpha)$ (for more information
about the Arveson spectrum of an action, see [1], [5])$.$ Since $M$ is
$\sigma-$ finite, it follows, in particular, that $M_{0}$\ has a faithful
normal state, $\varphi_{0}.$ This is the case, in particular, when $M_{0}$ has
separable predual.Then, $\varphi=\varphi_{0}\circ P_{0}^{M}$ is a faithful
normal $\alpha$-invariant state of $M$. Let $(H_{\varphi},\pi_{\varphi}%
,\xi_{\varphi})$ be the GNS representation of $M$\ corresponding to the state
$\varphi.$ Clearly, $\pi_{\varphi}$ is a faithful representation of $M$\ on
$H_{\varphi}$\ and $\xi_{\varphi}$\ is a cyclic and separating vector for $M.$
In the rest of the paper we will identify $\pi_{\varphi}(M)$ with $M$ and will
write $m$\ instead of $\pi_{\varphi}(m),m\in M$. Also, we will denote
$H_{\varphi}$\ by $H$\ and $\xi_{\varphi}$\ by $\xi_{0}.$ In other words, we
will consider a von Neumann algebra, $M\subset B(H),$ that has a cyclic and
separating unit vector $\xi_{0}$ and the vector state $\varphi(m)=\left\langle
m\xi_{0},\xi_{0}\right\rangle $ is $\alpha-$invariant. If we define
$U_{z}(m\xi_{0})=\alpha_{z}(m)\xi_{0},$ for all $z\in\boldsymbol{T},$ then it
is easy to see that $\left\{  U_{z}\right\}  _{z\in\boldsymbol{T}}$ is a group
of unitary operators on $H$ that implements the action $\alpha.$

If $S$\ and $F$\ are the closures of the Tomita-Takesaki conjugate linear
densely defined operators:%
\begin{align*}
S_{0}(m\xi_{0})  &  =m^{\ast}\xi_{0},m\in M\\
F_{0}(m^{\prime}\xi_{0})  &  =(m^{\prime})^{\ast}\xi_{0},m^{\prime}\in
M^{\prime}%
\end{align*}
then, it is known (see for instance [2], Section 9.2) that $S^{\ast}=F$, the
positive self adjoint operator $\Delta=S^{\ast}S$ has an inverse $\Delta
^{-1}=SS^{\ast}$ and there is a conjugate linear isometry $J$ from $H$\ onto
$H$\ such that%
\[
S=J\Delta^{\frac{1}{2}}%
\]

From the definition of the unitaries $U_{z},z\in\boldsymbol{T,}$ it
immediately follows that all $U_{z},z\in\boldsymbol{T,}$ commute with $S$ and
therefore with $S^{\ast}=F.$ Indeed%
\[
U_{z}S(m\xi_{0})=U_{z}(m^{\ast}\xi_{0})=\alpha_{z}(m^{\ast})\xi_{0}%
=SU_{z}(m\xi_{0})
\]
and therefore, $U_{z}S=SU_{z},z\in\boldsymbol{T.}$ Thus, all $U_{z}%
,z\in\boldsymbol{T}$ commute with $J$ as well.

We can now state the following

\begin{proposition}
With the notations above, we have the following:\newline i) $\alpha
_{z}^{\prime}(m^{\prime})=U_{z}m^{\prime}U_{z}^{\ast}$ is an action of
$\boldsymbol{T}$\ on $M^{\prime},$ where $M^{\prime}$\ is the commutant of
$M$\ in B(H)\newline ii) If $z\in\boldsymbol{T,}$ then$\ U_{z}$\ commutes with
$J$, and $J\alpha_{z}(m)J=\alpha_{z}^{\prime}(JmJ),m\in M,z\in\boldsymbol{T}%
$\newline iii) ($M^{\prime})_{n}=JM_{-n}J,n\in%
\mathbb{Z}
$\newline iv) $U_{z}(m^{\prime}\xi_{0})=\alpha_{z}^{\prime}(m^{\prime})\xi
_{0}$\newline v) $sp(\alpha)=sp(\alpha^{\prime})$
\end{proposition}

\begin{proof}
i) Indeed, if $m\in M$, $m^{\prime}\in M^{\prime}$ and $a\in M,$ we have:%
\[
U_{z}m^{\prime}U_{z}^{\ast}m(a\xi_{0})=U_{z}m^{\prime}U_{z}^{\ast}%
mU_{z}U_{\overline{z}}(a\xi_{0})==U_{z}m^{\prime}\alpha_{\overline{z}%
}(m)(\alpha_{\overline{z}}(a)\xi_{0})=
\]%
\[
U_{z}\alpha_{\overline{z}}(m)m^{\prime}(\alpha_{\overline{z}}(a)\xi
_{0})=v)sp(\alpha)=sp(\alpha^{\prime})=
\]%
\[
=mU_{z}m^{\prime}U_{\overline{z}}U_{z}(\alpha_{\overline{z}}(a)\xi_{0}%
)=mU_{z}m^{\prime}U_{z}^{\ast}(a\xi_{0})
\]
\newline Hence $\alpha_{z}^{\prime}(m^{\prime})=U_{z}m^{\prime}U_{z}^{\ast}\in
M^{\prime}.$\newline ii) The fact that the unitaries $U_{z}$\ commute with
$J$\ was proved above. For the second part of ii), notice that%
\[
J\alpha_{z}(m)J=JU_{z}mU_{z}^{\ast}J=U_{z}JmJU_{z}^{\ast}%
\]
iii) follows immediately from i) and ii).\newline iv) Let $m^{\prime}\in
M^{\prime}.$ Then, $m^{\prime}=JmJ$\ for some $m\in M.$ Then we have
\begin{align*}
U_{z}(m^{\prime}\xi_{0})  &  =U_{z}(JmJ\xi_{0})=JU_{z}(m\xi_{0})=\\
&  =J(\alpha_{z}(m)\xi_{0})=\alpha_{z}^{\prime}(m^{\prime})\xi_{0}%
\end{align*}
\newline v) Let $n\in sp(\alpha).$\ Then, since $M_{-n}=\left\{  m^{\ast}|m\in
M_{n}\right\}  $\ we have $-n\in sp(\alpha)$\ .From iii) it then follows that
$n\in sp(\alpha^{\prime}).$ The proof of the other inclusion is similar
\end{proof}

Let $H_{n}=\left\{
{\displaystyle\int}
\overline{z}^{n}U_{z}(\xi)dz|\xi\in H\right\}  =\left\{  \xi\in H|U_{z}%
\xi=z^{n}\xi\right\}  .$ Then, the map $P_{n}^{H}$ from $H$ to $H_{n}$ defined
as follows%
\[
P_{n}^{H}(\xi)=%
{\displaystyle\int}
\overline{z}^{n}U_{z}(\xi)dz|n\in%
\mathbb{Z}
,\xi\in H
\]
is an orthogonal projection of $H$\ onto the closed supspace $H_{n}.$ Also the
closed subspaces $(M^{\prime})_{n}\subset M^{\prime}$\ and the projections
$P_{n}^{M^{\prime}}$ from $M^{\prime}$\ onto $(M^{\prime})_{n}$\ can be
defined similarly with $M_{n}$\ and $P_{n}^{M}.$

\begin{proposition}
i) If $n\neq k,$ then $H_{n}$ and $H_{k}$ are orthogonal\newline ii) For every
$n\in%
\mathbb{Z}
$ we have $\overline{M_{n}\xi_{0}}=H_{n},$ where $M_{n}\xi_{0}=\left\{
m\xi_{0}|m\in M_{n}\right\}  $\newline iii) For every $n\in%
\mathbb{Z}
$ we have $\overline{(M^{\prime})_{n}\xi_{0}}=H_{n},$ where ($M^{\prime}%
)_{n}\xi_{0}=\left\{  m^{\prime}\xi_{0}|m^{\prime}\in(M^{\prime})_{n}\right\}
$\newline iv) The direct sum of Hilbert spaces $%
{\displaystyle\sum}
H_{n}$ equals $H$
\end{proposition}

\begin{proof}
i) Let $\xi\in H_{n},\eta\in H_{k}.$ Then $U_{z}(\xi)=z^{n}\xi$ and
$U_{z}(\eta)=z^{k}\eta,$ for all $z\in\boldsymbol{T.}$ Since $U_{z}$ are
unitary, we have%
\[
\left\langle \xi,\eta\right\rangle =\left\langle U_{z}\xi,U_{z}\eta
\right\rangle =z^{n-k}\left\langle \xi,\eta\right\rangle ,z\in\boldsymbol{T}%
\]
Since $n\neq k$ it follows that $\left\langle \xi,\eta\right\rangle
=0$\newline ii) Since $\xi_{0}$ is cyclic for $M,$ the subspace $M\xi
_{0}=\left\{  m\xi_{0}|m\in M\right\}  $\ is dense in $H.$\ Then, if
$P_{n}^{H}$ and $P_{n}^{M}$ are the above projections, we have
\[
M_{n}\xi_{0}=P_{n}^{M}(M)\xi_{0}=\left\{
{\displaystyle\int}
\overline{z}^{n}\alpha_{z}(m)(\xi_{0})dz|m\in M\right\}  =
\]%
\[
=\left\{
{\displaystyle\int}
\overline{z}^{n}U_{z}(\xi)dz|\xi=m\xi_{0},m\in M\right\}  =\left\{  P_{n}%
^{H}(\xi)|\xi=m\xi_{0},m\in M\right\}
\]
Since the subspace $M\xi_{0}$ is dense in $H$\ and $P_{n}^{H}$ is an
orthogonal projection, the result stated in ii) follows\newline iii) Since
$\xi_{0}$ is cyclic for $M^{\prime},$ the subspace $M^{\prime}\xi_{0}=\left\{
m^{\prime}\xi_{0}|m^{\prime}\in M^{\prime}\right\}  $\ is dense in $H.$ Then,
if $P_{n}^{H}$ and $P_{n}^{M}$ are the above projections, we have
\[
(M^{\prime})_{n}\xi_{0}=P_{n}^{M^{\prime}}(M^{\prime})\xi_{0}=\left\{
{\displaystyle\int}
\overline{z}^{n}\alpha_{z}^{\prime}(m^{\prime})(\xi_{0})dz|m^{\prime}\in
M^{\prime}\right\}  =
\]%
\[
=\left\{
{\displaystyle\int}
\overline{z}^{n}U_{z}(\xi)dz|\xi=m^{\prime}\xi_{0},m^{\prime}\in M\right\}
=\left\{  P_{n}^{H}(\xi)|\xi=m^{\prime}\xi_{0},m^{\prime}\in M^{\prime
}\right\}
\]
Since $\xi_{0}$ is cyclic for $M^{\prime}$\ , the result follows as in ii).
\newline iv) Let $\eta\in H$ be such that $\eta\perp H_{n}$ for all $n\in%
\mathbb{Z}
.$ Since $M$\ is the $w^{\ast}$-closed linear span of $\left\{  M_{n}|n\in%
\mathbb{Z}
\right\}  ,$\ it follows from ii) that $\eta\perp M\xi_{0},$ so, since
$\xi_{0}$ is cyclic for $M,$ $\eta=0.$
\end{proof}

\section{Invariant subspaces of subalgebras of analytic elements on
generalized Hardy spaces}

Let $(M,\boldsymbol{T},\alpha),M\subset B(H)$ and $\xi_{0}$ be as in the
previous section. If $sp(\alpha)=\left\{  0\right\}  ,$ then $\alpha$ is
trivial in the sense that $\alpha_{z}=id$ for every $z\in\boldsymbol{T},$
where $id$ is the identity automorphism, so $M=M_{0}$\ and $H=H_{0}$. Since
every von Neumann algebra is reflexive, this case has no interest from the
point of view of this paper. Suppose $sp(\alpha)\neq\left\{  0\right\}  $.
Then, by Proposition 1 v), we have $sp(\alpha^{\prime})=sp(\alpha)\neq\left\{
0\right\}  .$ We now define the generalized Hardy space, $H_{+}$ and the
subalgebras of analytic elements of $M,$\ $M_{+}\subset B(H_{+})$ and
($M^{\prime})_{+}\subset B(H_{+})$\ which are generalizations of the algebra
of analytic Toeplitz operators to the framework of periodic $W^{\ast}%
$-dynamical systems:%
\[
H_{+}=%
{\displaystyle\sum\limits_{n\geqslant0}}
H_{n}%
\]

If $p_{+}$ denotes the projection of $H$\ onto $H_{+}$, $\vee_{n\geqslant
0}M_{n}$ the $wo-$closed algebra generated by $\left\{  M_{n}|n\geqslant
0\right\}  $ and $\vee_{n\geqslant0}(M^{\prime})_{n}$\ the $wo-$closed algebra
generated by $\left\{  (M^{\prime})_{n}|n\geqslant0\right\}  $, we define%
\begin{align*}
M_{+}  &  =p_{+}(\vee_{n\geqslant0}M_{n})|_{H_{+}}\\
(M^{\prime})_{+}  &  =p_{+}(\vee_{n\geqslant0}(M^{\prime})_{n})|_{H_{+}}%
\end{align*}
In this section we will find conditions for the reflexivity of $M_{+}$
[Theorems 5, 8 and 13]. We consider first the case when $sp(\alpha)$\ is finite.

\begin{lemma}
If $sp(\alpha)\ $is finite$,$ then $(M^{\prime})_{+})^{\prime}=M_{+}$ \newline
\end{lemma}

\begin{proof}
Clearly, $M_{+}\subset((M^{\prime})_{+})^{\prime}.$ Let now $x\in(M^{\prime
})_{+}^{\prime}\subset B(H_{+}).$ Then $x$\ commutes, in particular, with
($M^{\prime})_{k},k\geqslant0.$ Consider the following dense linear
subspace\ of $H$
\[
H^{\prime}=\left\{  \xi\in H|P_{k}^{H}(\xi)\in(M^{\prime})_{k}\xi_{0}\text{
and }P_{k}^{H}(\xi)=0\text{ for all but finitely many }k\in%
\mathbb{Z}
\right\}
\]
Define the operator $\widehat{x}$ on $H^{\prime}$\ as follows
\[
\widehat{x}(m_{k}^{\prime}\xi_{0})=m_{k}^{\prime}x\xi_{0},k\in%
\mathbb{Z}
\]
Since $\xi_{0}$ is a separating vector for $M^{\prime}$\ and $x$\ commutes
with $(M^{\prime})_{+}$, $\widehat{x}$\ is well defined. Due to the obvious
facts that, for every operator $a\in B(H)$ and $\eta\in H,$\ $\left\Vert
a\eta\right\Vert =\left\Vert \left\vert a\right\vert \eta\right\Vert $ where
$\left\vert a\right\vert $\ is the absolute value of $a$, and $\left\vert
m_{k}^{\prime}\right\vert \in(M^{\prime})_{0}\subset(M^{\prime})_{+},$ it
follows that, if $m_{k}^{\prime}\in(M^{\prime})_{k},k\in%
\mathbb{Z}
,$ then
\begin{align*}
\left\Vert \widehat{x}m_{k}^{\prime}\xi_{0})\right\Vert  &  =\left\Vert
m_{k}^{\prime}x\xi_{0}\right\Vert =\left\Vert \left\vert m_{k}^{\prime
}\right\vert x\xi_{0}\right\Vert =\\
\left\Vert x(\left\vert m_{k}^{\prime}\right\vert \xi_{0})\right\Vert  &
\leq\left\Vert x\right\Vert \left\Vert \left\vert m_{k}^{\prime}\right\vert
\xi_{0}\right\Vert =\left\Vert x\right\Vert \left\Vert m_{k}^{\prime}\xi
_{0}\right\Vert
\end{align*}
Therefore $\widehat{x}$ is bounded on $H_{k},k\in sp(\alpha)$\ and $\left\Vert
\widehat{x}\right\Vert \leq\left\Vert x\right\Vert ,$\ so, since $sp(\alpha)$
is finite$,$ $\widehat{x}$\ is bounded on $H$. On the other hand, if $n,k\in%
\mathbb{Z}
,m_{n}^{\prime}\in(M^{\prime})_{n},m_{k}^{\prime}\in(M^{\prime})_{k}$ we have
$m_{n}^{\prime}m_{k}^{\prime}=m_{n+k}^{\prime}\in(M^{\prime})_{n+k}$, so%
\[
m_{n}^{\prime}\widehat{x}(m_{k}^{\prime}\xi_{0})=m_{n}^{\prime}m_{k}^{\prime
}x\xi_{0}=m_{n+k}^{\prime}x\xi_{0}=\widehat{x}(m_{n+k}^{\prime}\xi
_{0})=\widehat{x}(m_{n}^{\prime}m_{k}^{\prime}\xi_{0})
\]
Therefore, $\widehat{x}\in(M^{\prime})^{\prime}=M.$ Thus, $p_{+}\widehat
{x}p_{+}=x\in M_{+}$ and we are done.
\end{proof}

Suppose that $sp(\alpha)$\ is finite. In this case, since, as noticed above,
for every $n,k\in sp(\alpha),m_{n}\in M_{n},m_{k}\in M_{k}$ we have
$m_{n}m_{k}\in M_{n+k},$ it follows that every element of $M_{n},n\in
sp(\alpha)\backslash\left\{  0\right\}  $ is nilpotent. Notice also that, by
Proposition 1 v), we have $sp(\alpha)=sp(\alpha^{\prime}).$

\begin{lemma}
Let $n\in sp(\alpha)\backslash\left\{  0\right\}  ,n>0.$ and $w\in(M^{\prime
})_{n}$ a partial isometry. If $w$\ is nilpotent, then $w\in(algLat(M_{+}%
))^{\prime}$
\end{lemma}

\begin{proof}
We will use a method similar with the one used in the proof of [4, Proposition
19 i)]. Suppose $w^{k}=0$ for some $k\in%
\mathbb{N}
,k\geqslant2.$\ We will prove the Lemma by induction. Let first $k=2,$ so
$w^{2}=0.$ Let $p=1-w^{\ast}w\in(M^{\prime})_{0},$ be the orthogonal
projection of $H_{+}$ onto the kernel of $w.$ Then, $p,1-p\in Lat(M_{+})$ and
therefore, with respect to the decomposition $I_{H_{+}}=p+(1-p),$ we have
\begin{align*}
w  &  =\left(
\begin{array}
[c]{cc}%
0 & c\\
0 & 0
\end{array}
\right) \\
m  &  =\left(
\begin{array}
[c]{cc}%
m_{1} & 0\\
0 & m_{2}%
\end{array}
\right)  ,\forall m\in M_{+}\\
x  &  =\left(
\begin{array}
[c]{cc}%
x_{1} & 0\\
0 & x_{2}%
\end{array}
\right)  ,\forall x\in algLat(M_{+})
\end{align*}
Since $mw=wm,$ we have $m_{1}c=cm_{2}.$ Therefore the subspace $K=\left\{
c\xi\oplus\xi|\xi\in H_{+}\right\}  $ is invariant for $M_{+}$. Since $x\in
algLat(M_{+})$, it follows that $K\ $\ is invariant for $x$. Thus
$x_{1}c=cx_{2}$, so $wx=xw$ and we are done with the case $k=2.$ Suppose next
that for every partial isometry $w\in(M^{\prime})_{n}$ with $w^{k}=0$ it
follows that $w\in(algLat(M_{+}))^{\prime}$ and let $w\in(M^{\prime})_{n}$
with $w^{k+1}=0$. Let $p$ denotes the orthogonal projection on $\ker
w,p=1-w^{\ast}w,$ so that $p\in(M^{\prime})_{0}$. Since $wp=0$ it follows that
\newline%
\[
(1-p)w=(1-p)w(1-p)
\]
\newline and therefore\newline%
\[
(1-p)w^{k}=((1-p)w(1-p))^{k},k\in%
\mathbb{N}
.
\]
\newline Since $w^{k+1}=0$ we have $w^{k}(H_{+})\subset\ker w$ and
therefore\newline%
\[
0=(1-p)w^{k}=((1-p)w(1-p))^{k}.
\]
\newline By hypothesis, $(1-p)w=(1-p)w(1-p)\in(algLat$\textit{(}%
$M_{+}))^{\prime}.$ On the other hand, since $w\in(M^{\prime})_{1}$ and
$p_{0}\in(M^{\prime})_{0}$ we have $pw\in(M^{\prime})_{1}.$ Since obviously
$(pw)^{2}=0$, by the previous arguments for $k=2$, it follows that
$pw\in(algLat$\textit{(}$M_{+}))^{\prime}$. Therefore%
\[
w=pw+(1-p)w\in(algLat\text{\textit{(}}M_{+}))^{\prime}%
\]
and the proof is completed
\end{proof}

We can now state and prove the following

\begin{theorem}
Let $(M,\boldsymbol{T},\alpha)$ be as above. If $sp(\alpha)$ is finite, then
$M_{+}$ is reflexive.
\end{theorem}

\begin{proof}
Let $x\in algLat$\textit{(}$M_{+}).$ Since $pH_{+}\in Lat(M_{+})$\ for every
projection $p\in(M^{\prime})_{0}$, it follows that $x\in(M^{\prime}%
)_{0}^{\prime}.$ If $a\in(M^{\prime})_{n}$ for some $n>0,$ then $a$ has a
polar decomposition $a=w\left\vert a\right\vert $, where obviously
$w\in(M^{\prime})_{n}$ and $\left\vert a\right\vert \in(M^{\prime})_{0}.$
Since $sp(\alpha)$ is finite, $w$\ is nilpotent and terefore by the previous
lemma $xw=wx,$ so $x\in(M^{\prime})_{+}^{\prime}.$Applying Lemma 3 it follows
that $x\in M_{+}$ and we are done.
\end{proof}

\begin{example}
Let $M$\ be a von Neumann algebra and $\left\{  p_{1},p_{2},...,p_{k}\right\}
$ a family of mutually orthogonal projections of $M.$ Let $U_{z}=%
{\displaystyle\sum}
z^{n_{i}}p_{i},z\in\boldsymbol{T},$ where $\left\{  n_{1},n_{2},...,n_{k}%
\right\}  \subset%
\mathbb{Z}
$ is a decreasing finite set. Consider the periodic flow $\alpha_{z}%
(m)=U_{z}mU_{z}^{\ast}$,$z\in\boldsymbol{T}$ on $M.$ It is easy to show that
\[
M_{+}=\left\{  m\in M|m(%
{\displaystyle\sum\nolimits_{j=1}^{i}}
p_{j}H)\subset%
{\displaystyle\sum\nolimits_{j=1}^{i}}
p_{j}H,i=1,...,k\right\}
\]
is the reflexive nest subalgebra of $M$\ corresponding to the nest $p_{1}\leq
p_{1}+p_{2}\leq...\leq%
{\displaystyle\sum\nolimits_{j=1}^{k}}
p_{j}$ and $sp(\alpha)$ is finite.
\end{example}

Let $n_{0}=\min\left\{  n\in sp(\alpha)|n>0\right\}  .$ Suppose that
$M_{n_{0}}$ contains a unitary operator, $v_{0}.$

\begin{proposition}
In the above conditions, we have\newline i) $M_{kn_{0}}=\left\{  v_{0}%
^{k}a|a\in M_{0}\right\}  ,k\in%
\mathbb{Z}
$\newline ii) $H_{kn_{0}}=\overline{\left\{  v_{0}^{k}a\xi_{0}|a\in
M_{0}\right\}  },k\in%
\mathbb{Z}
$\newline iii) ($M^{\prime})_{n_{0}}$ contains an unitary operator, $w_{0}%
$\newline iv) ($M^{\prime})_{kn_{0}}=\left\{  w_{0}^{k}b|b\in(M^{\prime}%
)_{0}\right\}  ,k\in%
\mathbb{Z}
$\newline v) $sp(\alpha)$ is a subgroup of $%
\mathbb{Z}
$
\end{proposition}

\begin{proof}
Let $m\in M_{kn_{0}}.$ Then, ($v_{0}^{\ast})^{k}m=a\in M_{0}$\newline ii)
follows from i) and Proposition 2, ii) \newline iii) By Proposition 1, iii)
the unitary operator $w_{0}=Jv_{0}^{\ast}J$ is in ($M^{\prime})_{n_{0}}%
$\newline iv) Similar with i)\newline v) From i) it follows that $\left\{
kn_{0}|k\in%
\mathbb{Z}
\right\}  \subseteq sp(\alpha)$. On the other hand, if $l\in%
\mathbb{Z}
,l\in sp(\alpha)$, then there exists $k_{0}\in%
\mathbb{Z}
$ such that $k_{0}n_{0}\leq l\leq(k_{0}+1)n_{0}.$ It can be easily verified
that, ($v_{0}^{\ast})^{k}\in M_{kn_{0}}^{\ast}=M_{-kn_{0}}$ and, since $v_{0}$
is unitary, we have ($v_{0}^{\ast})^{k_{0}}M_{l}\neq\left\{  0\right\}  .$ But
($v_{0}^{\ast})^{k_{0}}M_{l}\subseteq M_{l-k_{0}n_{0}},$ so $l-k_{0}n_{0}\in
sp(\alpha).$ Since $l-k_{0}n_{0}\leq n_{0}$ and $n_{0}=\min\left\{  n|n\in
sp(\alpha),n>0\right\}  $, it follows that either $l=k_{0}n_{0}$ or
$l=(k_{0}+1)n_{0}.$ Therefore, $sp(\alpha)=\left\{  kn_{0}|k\in%
\mathbb{Z}
\right\}  $ and we are done
\end{proof}

\begin{theorem}
Let $(M,\boldsymbol{T},\alpha),H_{+,}M_{+},v_{0}$ be as above. Then, the
algebra $M_{+}\subset B(H_{+})$\ is reflexive.
\end{theorem}

The proof of this result will be given after a series of lemmas. The next
lemma is a substitute for Lemma 3 for the case when $M_{n_{0}},$ (and
therefore ($M^{\prime})_{n_{0}}),$ contains an unitary operator $v_{0}$
(respectively $w_{0})$.

\begin{lemma}
i) $(M_{+})^{\prime}=(M^{\prime})_{+}$\newline ii) $((M^{\prime})_{+}%
)^{\prime}=M_{+}$
\end{lemma}

\begin{proof}
i) Let $x\in(M^{\prime})_{n}$ and $m\in M_{k},n,k\geqslant0$\ . Then
$p_{+}xp_{+}mp_{+}=p_{+}xmp_{+}=p_{+}mxp_{+}=p_{+}mp_{+}xp_{+},$ so
$(M^{\prime})_{+}\subseteq(M_{+})^{\prime}.$ Let now $x\in(M_{+})^{\prime
}\subset B(H_{+}).$ Consider the following linear subspace\ of $H$
\[
H^{\prime}=\left\{  \xi\in H|P_{n}^{H}(\xi)\in M_{n}\xi_{0},n\in%
\mathbb{Z}
\text{ and }P_{n}^{H}(\xi)=0\text{ for all but finitely many }n\right\}
\]
Using Proposition 2 iii) and Proposition 7 iv), it follows that $H^{\prime}%
$\ is a dense subspace of $H.$\ Let $\xi\in H^{\prime}\cap H_{+}.$ Define the
following operator, $\widehat{x}$ on $H^{\prime}$%
\[
\widehat{x}((w_{0}^{\ast})^{n}\xi)=(w_{0}^{\ast})^{n}x\xi,n>0
\]
Then, $\widehat{x}$ is well defined. Indeed if $\xi\in H_{+}$ and $\xi
^{\prime}\in H_{+}$\ are such that $(w_{0}^{\ast})^{n}\xi=(w_{0}^{\ast}%
)^{k}\xi^{\prime}$ for $n,k\in%
\mathbb{Z}
,$ say $n\leq k,$ then, $(w_{0}^{\ast})^{n-k}\xi=(w_{0})^{k-n}\xi=\xi^{\prime
}.$ Since $(w_{0})^{k-n}\in M_{+},$ and $x$\ commutes with $M_{+}$\ we have
\[
(w_{0}^{\ast})^{n-k}x\xi=(w_{0})^{k-n}x(\xi)=x(w_{0})^{k-n}\xi=x\xi^{\prime}%
\]
Therefore,%
\[
(w_{0}^{\ast})^{n}x\xi=(w_{0}^{\ast})^{k}x\xi^{\prime}%
\]
so
\[
\widehat{x}((w_{0}^{\ast})^{n}\xi)=\widehat{x}((w_{0}^{\ast})^{k}\xi^{\prime
})
\]
Next we prove that the operator $\widehat{x}$ is bounded. Indeed%
\[
\left\Vert \widehat{x}((w_{0}^{\ast})^{n}\xi)\right\Vert =\left\Vert
(w_{0}^{\ast})^{n}x\xi\right\Vert =\left\Vert x\xi\right\Vert \leq\left\Vert
x\right\Vert \left\Vert \xi\right\Vert =\left\Vert x\right\Vert \left\Vert
(w_{0}^{\ast})^{n}\xi\right\Vert
\]
Since $w_{0}\in M^{\prime},$ it is straightforward to prove that $\widehat
{x}\in M^{\prime}$ and $p_{+}\widehat{x}p_{+}=x.$ Thus $x\in(M^{\prime})_{+}$
and we are done\newline ii) By replacing $M$\ with $M^{\prime}$\ in i), and
using Proposition 2, iii), the result follows
\end{proof}

For every $\lambda\in%
\mathbb{C}
,\left\vert \lambda\right\vert <1,$ let $K_{\lambda}=\left\{  x(\lambda,b)=%
{\displaystyle\sum\limits_{n\geqslant0}}
\lambda^{n}w_{0}^{n}b|b\in(M^{\prime})_{0}\right\}  $ and $K=\overline
{lin\left\{  K_{\lambda}|\lambda\in%
\mathbb{C}
,\left\vert \lambda\right\vert <1\right\}  },$ where the closure is in the
norm topolgy of $M^{\prime}.$

\begin{lemma}
With the abve notations, $(M^{\prime})_{0}\subset K$ and $w_{0}\in K$.
\end{lemma}

\begin{proof}
Clearly, for $\lambda=0$, we get $K_{0}=(M^{\prime})_{0}\subset K.$ We will
prove now that $w_{0}^{n}\in K$ for every $n\in%
\mathbb{N}
.$ Fix $n\in%
\mathbb{N}
$ and let $0<\epsilon<1.$ Denote by $\lambda_{0},\lambda_{1},...,\lambda_{n}$
the complex roots of the equation $\lambda^{n+1}=\epsilon^{n+1}.$ Let
$\left\{  \mu_{0},\mu_{1},...,\mu_{n}\right\}  $ be the solution of the
following system of linear equations:%
\[%
{\displaystyle\sum\limits_{i=0}^{i=n}}
\mu_{i}\lambda_{i}^{j}=\delta_{jn},j=0,1,...,n
\]
\newline where $\delta_{jn}$\ is the Kronecker symbol. Using Cramer's rule, we
get%
\[
\mu_{i}=\frac{\pm1}{\Pi_{k<i}(\lambda_{k}-\lambda_{i})\Pi_{k>i}(\lambda
_{i}-\lambda_{k})}%
\]
\newline The choices of $\left\{  \lambda_{i}\right\}  $ and $\left\{  \mu
_{i}\right\}  $ imply that $%
{\displaystyle\sum\limits_{i=0}^{n}}
\mu_{i}x(\lambda_{i},1)=w_{0}^{n}+%
{\displaystyle\sum\limits_{j=n+1}^{\infty}}
(%
{\displaystyle\sum\limits_{i=0}^{n}}
\mu_{i}\lambda_{i}^{j})w_{0}^{j}.$ Furthermore, for $i\neq k,k\geqslant1$ we
have
\[
\left\vert \lambda_{i}-\lambda_{k}\right\vert \geqslant\left\vert
\lambda_{k+1}-\lambda_{k}\right\vert =\left\vert \lambda_{k}-\lambda
_{k-1}\right\vert =2\epsilon\sin\frac{\pi}{n+1}%
\]
\newline Therefore%
\[
\left\vert \mu_{i}\right\vert =\frac{1}{2^{n}\epsilon^{n}(\sin\frac{\pi}%
{n+1})^{n}},i=0,1,...,n
\]
\newline Thus
\[
\left\vert
{\displaystyle\sum\limits_{i=0}^{n}}
\mu_{i}\lambda_{i}^{j}\right\vert \leq\frac{\epsilon^{j}}{2^{n}\epsilon
^{n}(\sin\frac{\pi}{n+1})^{n}},j\geqslant n+1
\]
\newline It follows that%
\[
\left\Vert
{\displaystyle\sum\limits_{i=0}^{n}}
\mu_{i}x(\lambda_{i},1)-w_{0}^{n}\right\Vert =\left\Vert
{\displaystyle\sum\limits_{j=n+1}^{\infty}}
(%
{\displaystyle\sum\limits_{i=0}^{n}}
\mu_{i}\lambda_{i}^{j})w_{0}^{j}\right\Vert \leq%
{\displaystyle\sum\limits_{j=n+1}^{\infty}}
\frac{\epsilon^{j}}{2^{n}\epsilon^{n}(\sin\frac{\pi}{n+1})^{n}}=\frac
{\epsilon}{2^{n}(\sin\frac{\pi}{n+1})^{n}(1-\epsilon)}%
\]
\newline Hence $w_{0}\in K$ and we are done.
\end{proof}

For $\lambda\in%
\mathbb{C}
,\left\vert \lambda\right\vert <1,$ denote
\[
\widetilde{H}_{+}^{\lambda}=\left\{
{\displaystyle\sum\limits_{n=0}^{\infty}}
\lambda^{n}w_{0}^{n}\xi|\xi\in H_{0}\right\}
\]

\begin{corollary}
The linear subspace, $lin\left\{  \widetilde{H}_{+}^{\lambda}|\lambda\in%
\mathbb{C}
,\left\vert \lambda\right\vert <1\right\}  ,$ spanned by $\left\{
\widetilde{H}_{+}^{\lambda}|\lambda\in%
\mathbb{C}
,\left\vert \lambda\right\vert <1\right\}  $ is dense in $H_{+}.$
\end{corollary}

\begin{proof}
By Proposition 2 ii), $H_{0}=\overline{(M^{\prime})_{0}\xi_{0}}$ and
therefore, by Proposition 7 iv), $H_{kn_{0}}=\overline{w_{0}^{k}H_{0}}.$ The
result follows from the above lemma
\end{proof}

We can give now the

\begin{proof}
of Theorem 8. By Proposition 2 iii), ($M^{\prime})_{0}p_{+}\subset B(H_{+}).$
Since $v_{0}\in M_{n_{0}}\subset M$, it follows that $v_{0}e=ev_{0}$ for every
projection $e\in$($M^{\prime})_{0}$. Therefore, for every projection $e\in
$($M^{\prime})_{0},$ $eH_{+}\in Lat(M_{+}).$ Let $x\in algLat(M_{+})$. Then,
in particular, $xe=ex$ for every projection $e\in$($M^{\prime})_{0},$ hence,
$x^{\ast}e=ex^{\ast}$ for every projection $e\in$($M^{\prime})_{0},$ where
$x^{\ast}\ $is the adjoint of $x$ in $B(H_{+}).$ It is clear that $x^{\ast}\in
algLat(M_{+}^{\ast}),$ where $M_{+}^{\ast}\subset B(H_{+})$\ is the algebra of
the adjoint elements of $M_{+}\subset B(H_{+}).$ On the other hand, since
$v_{0}e=ev_{0}$ for all projections $e\in$($M^{\prime})_{0},$ we have that
$v_{0}^{\ast}e=ev_{0}^{\ast}$ for $e\in$($M^{\prime})_{0}$ and $v_{0}^{\ast
}\in M_{-n_{0}}\subset B(H).$ If we denote by $S_{v_{0}}$ the operator $v_{0}%
$\ on $H_{+},$\ we have%
\begin{align*}
S_{v_{0}}^{\ast}(\xi)  &  =0\text{ if }\xi\in H_{0}\text{ and}\\
S_{v_{0}}^{\ast}(\xi)  &  =v_{0}^{\ast}\xi\text{ if }\xi\in H_{n},n>0
\end{align*}
\newline It follows that $S_{v_{0}}^{\ast}e=eS_{v_{0}}^{\ast}$ for all $e\in
$($M^{\prime})_{0}.$ Hence $eH_{+}\in Lat(M_{+}^{\ast})$ for all $e\in
$($M^{\prime})_{0}.$ On the other hand, by the definition of $w_{0}$, we have
$v_{0}w_{0}=w_{0}v_{0}$. Thus%
\[
S_{v_{0}}^{\ast}S_{w_{0}}^{\ast}=S_{w_{0}}^{\ast}S_{v_{0}}^{\ast}\text{ and
}aS_{w_{0}}^{\ast}=S_{w_{0}}^{\ast}a,a\in M_{0}%
\]
If we denote%
\[
H_{+}^{\lambda}=\left\{  \zeta\in H_{+}|S_{w_{0}}^{\ast}\zeta=\lambda
\zeta\right\}
\]
it is clear that
\[
\widetilde{H}_{+}^{\lambda}\subset H_{+}^{\lambda}\text{ for all }\lambda\in%
\mathbb{C}
,\left\vert \lambda\right\vert <1
\]
Since for every $\lambda\in%
\mathbb{C}
,\left\vert \lambda\right\vert <1,$ we have $H_{+}^{\lambda}\in Lat(M_{+}%
^{\ast}),$ where $H_{+}^{\lambda}$ denotes, as above, the closed subspace of
all eigenvectors of $S_{w_{0}}^{\ast}$ corresponding to the eigenvalue
$\lambda,$\ and $x^{\ast}\in algLat(M_{+}^{\ast}),$\ it follows that $x^{\ast
}H_{+}^{\lambda}\subset H_{+}^{\lambda},$ for all $\lambda\in%
\mathbb{C}
,\left\vert \lambda\right\vert <1.$ We will prove next that
\[
x^{\ast}S_{w_{0}}^{\ast}=S_{w_{0}}^{\ast}x^{\ast}%
\]
and then apply Lemma 9 to conclude that $x\in M_{+}.$ Since $x^{\ast}%
H_{+}^{\lambda}\subset H_{+}^{\lambda},$ it immediately follows that $x^{\ast
}S_{w_{0}}^{\ast}=S_{w_{0}}^{\ast}x^{\ast}$ on $H_{+}^{\lambda}.$ By Corollary
12, the linear span of $\left\{  \widetilde{H}_{+}^{\lambda}|\lambda\in%
\mathbb{C}
,\left\vert \lambda\right\vert <1\right\}  $ is dense in $H_{+}.$ Since
$\widetilde{H}_{+}^{\lambda}\subset H_{+}^{\lambda},$ we have that the linear
span of $\left\{  H_{+}^{\lambda}|\lambda\in%
\mathbb{C}
,\left\vert \lambda\right\vert <1\right\}  $ is dense in $H_{+}.$ Therefore
$x^{\ast}S_{w_{0}}^{\ast}=S_{w_{0}}^{\ast}x^{\ast}$ on $H_{+}.$ Hence
$x\in((M^{\prime})_{+})^{\prime}.$ From Lemma 9 ii) it follows that $x\in
M_{+}$ and we are done.
\end{proof}

\begin{remark}
Let $A$ be a weakly closed algebra of operators on the Hilbert space $X.$
Suppose that there exists a family of mutually orthogonal projections
$\left\{  q_{\iota}\right\}  \subset A$ such that:\newline i) $q_{\iota}$
commutes with $A$ for every $\iota$\newline ii) $q_{\iota}Aq_{\iota}\subset
B(q_{\iota}X)$ is reflexive for ever $\iota$ and\newline iii) $%
{\displaystyle\sum}
q_{\iota}=I_{X}$ where $I_{X}$\ is the identity of $B(X)$\newline Then
$A\subset B(X)$\ is reflexive.
\end{remark}

\begin{proof}
Let $x\in algLat(A).$ Since $q_{\iota},1-q_{\iota}\in Lat(A),$ we have
$xq_{\iota}=q_{\iota}xq_{\iota}\in algLat(q_{\iota}Aq_{\iota}).$ The
reflexivity of $q_{\iota}Aq_{\iota}$\ implies that $xq_{\iota}\in q_{\iota
}Aq_{\iota}=Aq_{\iota}$ for all $\iota.$ Therefore $x=%
{\displaystyle\sum}
xq_{\iota}\in A.$
\end{proof}

The above techniques allow us to prove the following

\begin{theorem}
Let ($M,\boldsymbol{T},\alpha)$ be a $W^{\ast}-$dynamical system with
$M$\ abelian in the standard form considered above. Then $M_{+}$\ is reflexive.
\end{theorem}

\begin{proof}
Let $n_{0}=\min\left\{  n\in sp(\alpha)|n>0\right\}  .$ Let $v\in M_{n_{0}}$
be a partial isometry and $e=v^{\ast}v=vv^{\ast}\in M_{0}.$ Then, $ve$\ is an
unitary element of $eM_{n_{0}}e\subset B(eH_{+}).$ By Theorem 8, we have that
$eM_{+}e$ is a reflexive subalgebra of $B(eH_{+}).$ If $(1-e)M_{n_{0}%
}(1-e)\neq\left\{  0\right\}  ,$ then there is a partial isometry $u\in
M_{n_{0}}$ such that $u^{\ast}u=uu^{\ast}\leq1-e.$\ Let $\left\{  e_{\gamma
}\right\}  _{\gamma\in\Gamma_{1}}\subset M_{0}$ be a maximal family of
orthogonal projections such that for every $\gamma\in\Gamma_{1},$ $e_{\gamma
}=v_{\gamma}v_{\gamma}^{\ast}=v_{\gamma}^{\ast}v_{\gamma}$ for some partial
isometry $v_{\gamma}\in M_{n_{0}}.$ Set $f_{1}=%
{\displaystyle\sum\nolimits_{\gamma\in\Gamma_{1}}}
e_{\gamma}.$ Then,  by Remark 12,  $f_{1}M_{+}f_{1}$ is a reflexive subalgebra
of $B(f_{1}H_{+})$. Clearly,%
\[
(1-f_{1})M_{n_{0}}(1-f_{1})=\left\{  0\right\}
\]

Suppose, by induction, that for every $k\in%
\mathbb{N}
,k\geqslant2$\ there is a projection $f_{k}\in M_{0},f_{k}\leq1-\sum
_{j=1}^{k-1}f_{j},$ such that
\[
f_{k}M_{+}f_{k}\text{ is a reflexive subalgebra of }B(f_{k}H_{+})\text{ and}%
\]%
\[
(1-\sum_{j=1}^{k}f_{j})M_{i}(1-\sum_{j=1}^{k}f_{j})=\left\{  0\right\}  \text{
for all }i,n_{0}\leq i<n_{0}+k
\]
Then, as in the step $k=1$, we can find $f_{k+1}\leq1-%
{\displaystyle\sum\nolimits_{i=1}^{i=k}}
f_{i}$ such that%
\[
f_{k+1}M_{+}f_{k+1}\text{ is a reflexive subalgebra of }B(f_{k+1}H_{+})\text{
and}%
\]%
\[
(1-\sum_{j=1}^{k+1}f_{j})M_{i}(1-\sum_{j=1}^{k+1}f_{j})=\left\{  0\right\}
\text{ for all }i,n_{0}\leq i<n_{0}+k+1
\]
Let $f=%
{\displaystyle\sum\nolimits_{j\geqslant1}}
f_{j}.$ Then, applying again Remark 12 it follows that $fM_{+}f$ is a
reflexive subalgebra of $B(fH_{+})$ and ($1-f)M_{i}(1-f)=\left\{  0\right\}  $
for all $i\in%
\mathbb{N}
.$ Therefore,%
\[
(1-f)M_{+}(1-f)=(1-f)M_{0}(1-f)
\]
which is a von Neumann subalgebra of $B((1-f)H_{+})$ and is ttherefore
reflexive. Thus, by Remark 12 we have that%
\[
M_{+}=fM_{+}\oplus(1-f)M_{+}%
\]
is reflexive
\end{proof}

\bigskip

\ \ \ \ \ \ \ \ \ \ \ \ \ \ \ \ \ \ \ \ \ \ \ \ \ \ \ \ \ \textbf{REFERENCES}

\bigskip

1. W. B. Arveson, The harmonic analysis of automorphism groups, operator
algebras and applications, Proc. Sympos. Pure Math., vol. 38, Amer. Math.
Soc., Providence, R. I., 1982.

2. R. V. Kadison and J. R. Ringrose, Fundamentals of the theory of operator
algebras, Vol. II Advanced Theory,Academic Press 1986.

3. E. T. A. Kakariadis, Semicrossed products and reflexivity, J. Operator
Theory, 67(2012), 379-395.

4. F Merlevede, C, Peligrad and M. Peligrad,
Reflexive\ Operator\ Algebras\ on\ Banach\ Spaces, Pacific J. Math., 267
(2014), 451-464, DOI: 10.2140/pjm.2014.267.451.

5. G. K. Pedersen, C*-algebras and their automorphism groups, Academic Press 1979.

6. C. Peligrad, Reflexive operator algebras on noncommutative Hardy spaces,
Math. Annalen, 253(1980), 165-175.

7. H. Radjavi and P. Rosenthal,\ Invariant subspaces, 2nd edition, Dover
Publications, Mineola, New York, 2003\textit{.}

8. D. Sarason, Invariant subspaces and unstarred operator algebras, Pacific J.
Math., 17(1966), 511-517.

\end{document}